\documentclass[12pt, reqno]{amsart}
\usepackage{amsmath, amsthm, amscd, amsfonts, amssymb, graphicx, color, mathrsfs}
\usepackage[bookmarksnumbered, colorlinks, plainpages]{hyperref}
\usepackage[all]{xy}
\usepackage{slashed}

\usepackage{soul}
\usepackage{cancel}
\usepackage{ulem}

\newtheorem{theorem}{Theorem}[section]

\newtheorem{proposition}[theorem]{Proposition}
\newtheorem{corollary}[theorem]{Corollary}
\theoremstyle{definition}
\newtheorem{definition}[theorem]{Definition}

\theoremstyle{remark}

\numberwithin{equation}{section}

\begin{document}
\setcounter{page}{1}

\title[  Dixmier traces for discrete pseudo-differential operators ]{ Dixmier traces for discrete pseudo-differential operators}

\author[D. Cardona]{Duv\'an Cardona}
\address{
		Duv\'an Cardona:
		\endgraf
		Department of Mathematics: Analysis Logic and Discrete Mathematics
		\endgraf
		Ghent University
		\endgraf
		Ghent-Belgium
		\endgraf
		{\it E-mail address} {\rm duvanc306@gmail.com}
		}

\author[C. del Corral]{C\'esar del Corral}
\address{
  C\'esar del Corral:
  \endgraf
  Department of Mathematics  
  \endgraf
  Universidad de los Andes
  \endgraf
  Bogot\'a
  \endgraf
  Colombia
  \endgraf
  {\it E-mail address} {\rm ce-del@uniandes.edu.co; cesar.math@gmail.com}
  }

\author[V. Kumar]{Vishvesh Kumar}
\address{
		Vishvesh Kumar:
		\endgraf
		Department of Mathematics: Analysis Logic and Discrete Mathematics
		\endgraf
		Ghent University
		\endgraf
		Ghent-Belgium
		\endgraf
		{\it E-mail address} {\rm vishveshmishra@gmail.com}
		\endgraf
	}

\dedicatory{Dedicated to Professor L\'azaro Recht on his 79 birthday}

\subjclass[2010]{Primary {}.}


\begin{abstract}
In this paper we provide sharp results  for the Dixmier traceability of discrete pseudo-differential operators on  $\ell^2(\mathbb{Z}^n)$. In this setting, we introduce a suitable notion of a class of classical symbols which provide a class of Dixmier traceable discrete pseudo-differential operators. We also present a formula for the Dixmier trace of a Dixmier traceable discrete pseudo-differential operator  by using the Connes equivalence between the Wodzicki residue and the Dixmier trace. 
\end{abstract} \maketitle

\section{Introduction}
In this paper we study the membership of global pseudo-differential operators  on $\mathbb{Z}^n,$ also known as discrete pseudo-differential operators, associated with the discrete H\"ormander symbol classes   introduced by Botchway, Kibity and Ruzhansky  \cite{Ruzhansky} to the  Dixmier ideal of operators on $\ell^2(\mathbb{Z}^n).$ For this purpose, we use a fundamental tool of non-commutative geometry, namely, the celebrated theorem of Connes on the equivalence of the Wodzicki residue and the Dixmier trace \cite{Connes88}. Recently, due to its applications in  number theory, difference equations  and many other problems involving the discretization of differential equations, the theory of pseudo-differential operators on $\mathbb{Z}^n$ has gained an important role in harmonic analysis \cite{Pie2,st3,Hooley,IW,Pie,st,st2,st3}. Discrete pseudo-differential operators appear early in the classic Fourier analysis (see \cite{Duo,GrafakosBook}) as well as discrete counterparts of Calder\'on-Zygmund singular integral operators. 

 The question ``whether the $C^*$-algebra  of all
bounded linear operators on a Hilbert space $H$ had a unique non-trivial trace" led to the concept of 
the Dixmier trace $\textnormal{Tr}_\omega.$ In 1966, Dixmier answered this question in negative in a note  published in Comptes Rendus \cite{Dix66}. To obtain the trace he used an invariant mean $\omega$ on the
 $ax+b$ group. In contrast to the usual trace, his trace vanishes on the ideal of trace class operators.
Applications of the Dixmier trace in differential  geometry as well as in non-commutative geometry were facilitated by the celebrated theorem of  Connes \cite{Connes88} which relates the trace to the Wodzicki  residue \cite{Wodzicki} for pseudo-differential operators on a closed manifold. Both the Dixmier trace and
the Wodzicki residue play important roles in noncommutative geometry and
its applications to fractal theory,  foliation theory, spaces of non-commuting coordinates, perturbation theory, and quantum field theory \cite{Connes88}.

In a recent context, some propeties of the pseudo-differential operators on $\mathbb{Z}^n,$ studied  by Molahajloo in \cite{m} which were initially introduced by Rabinovich and Rabinovich and Roch \cite{RabRoch, Rab3} but without the symbolic calculus. But still the symbolic calculus for this quantization was missing. More recently,  the symbol classes and symbolic calculus  have been introduced and  constructed in the work of Botchway,  Kibiti, and  Ruzhansky \cite{Ruzhansky} by asking the symbolic condition for decay in the space variable rather than frequency variable. They also presented  its  applications to difference equations, as well as the G\r{a}rding and the sharp G\r{a}rding inequality on $\mathbb{Z}^n$. In this setting the pseudo-differential operators on $\mathbb{Z}^n$ are defined by 
\begin{equation}\label{pseudo}
  t_m f(n'):=\int\limits_{\mathbb{T}^n} e^{i2\pi n'\cdot\xi}m(n',\xi)(\mathscr{F}f)(\xi)d\xi,\,\,\,f\in \mathscr{S}(\mathbb{Z}^n),\,n'\in\mathbb{Z}^n.
\end{equation}
Under suitable conditions on the symbol $m,$ we have that $t_m$ is a continuous linear operator from the discrete Schwartz space $\mathscr{S}(\mathbb{Z}^n)$ into itself. This is the case of symbols in the discrete H\"ormander classes (see e.g. \cite{Ruzhansky}).
The references Rabinovich\cite{Rab1,Rab2}, and Rabinovich and Roch\cite{Rab3,Rab4} are predecessor works on the subject.  Analytic and spectral properties as the ellipticity and Fredholmness, nuclearity, the spectral traces, and the $L^{p}$ and weak-$L^p$-boundedness  for these operators have been considered in Botchway,  Kibiti, and  Ruzhansky \cite{Ruzhansky}, Catana \cite{Cat14}, Wong\cite{Wong2},  Rodriguez \cite{rod}, Dasgupta and Kumar \cite{Aparajita19},  and  the works of the first and third author \cite{KC2018}. In particular, the Calder\'on-Vaillancourt theorem and the Gohberg Lemma for the discrete H\"ormander classes in \cite{Ruzhansky} have been proved in Cardona \cite{cardonazn}. 

In this paper we are interested in providing sharp conditions on $m,$ in order that $t_m$ in \eqref{pseudo}, belongs to the Dixmier ideal $\mathcal{L}^{1,\infty}(\ell^2(\mathbb{Z}^n)).$
The Dixmier traceability of pseudo-differential operators on smooth manifolds is a classical problem. Indeed, the problem was treated by using the notion of local symbols in  the works of Connes, Dixmier, Fedosov, Grubb, Schrohe, \cite{Connes,Fedosov,GGIS,Schrohe2,Schrohe,Wodzicki} and references therein. We refer the reader to \cite{Cd2} in which the Dixmier traceability of pseudo-differential operators on compact manifolds has been studied by using the matrix-valued quantization (see \cite{Ruz}).

Our main result is Theorem \ref{mainTheorem}, where we prove that, under the condition 
 $t_m\in \Psi^{-n}_{cl}(\mathbb{Z}^n)$ (see Definition \ref{CZN})  $t_m$ being a positive pseudo-differential operator, its Dixmier trace is given, surprisingly,  by the formula
\begin{equation}\label{maineq}
\textnormal{Tr}_\omega(t_m)= \frac{1}{n(2\pi)^n} \int\limits_{\mathbb{T}^n}\int\limits_{\mathbb{S}^{n-1}}[i^*(\tau)(x,\xi)]_{(-n)}d\Sigma(\xi) d_{\textnormal{Vol}}x,\,\,\tau(x,k):=\overline{m(-k,x)},
\end{equation} where $d\Sigma(\xi)$ is the surface measure associated to the sphere $\mathbb{S}^{n-1}\subset \mathbb{R}^n,$ and $d_{\textnormal{Vol}}x$ is the volume form on $M=[0,1)^n.$  The integral formula  \eqref{maineq} shows that we can capture the spectral information in  the Dixmier trace of an operator in terms of the geometry of the manifold $\mathbb{T}^n\times \mathbb{S}^{n-1},$ where the density $d\Sigma(\xi) d_{\textnormal{Vol}}x$ appears, together with the  Fourier analysis encoded in the global symbol $i^{*}\tau=i^*(\overline{m(-\cdot,\cdot)})$.

In order to prove this formula, we will provide some preliminaries in Section \ref{preliminaries} and later we will present the proof in Section \ref{proof}.

\section{Pseudo-differential calculus on $\mathbb{Z}^n$ and the $n$-torus $\mathbb{T}^n$}\label{preliminaries}

In this paper we will use the usual notation and terminology given in the standard monographs or papers (see e.g. \cite{f, Hor1, Hor2, Ruz, Wong, Wong2, Rt}). The discrete Schwartz space $\mathscr{S}(\mathbb{Z}^n)$ is the  the space of all functions $\phi:\mathbb{Z}^n\rightarrow \mathbb{C}$ such that 
 \begin{equation}
 \forall M\in\mathbb{R},\, \exists\, C_{M}>0,\, \text{such that}\, |\phi(\xi)|\leq C_{M}\langle \xi \rangle^M,
 \end{equation}
where $\langle \xi \rangle=(1+|\xi|^2)^{\frac{1}{2}}.$ The toroidal Fourier transform is defined as follows:  for any $f\in C^{\infty}(\mathbb{T}^n),$   $$ \widehat{f}(\xi)=\int\limits_{\mathbb{T}^n}e^{-i2\pi\langle x,\xi\rangle}f(x)dx,\,\,\xi\in\mathbb{Z}^n.$$ The inverse toroidal Fourier transform is given by the following formula: for $f \in C^\infty(\mathbb{T}^n)$ $$f(x)=\sum\limits_{\xi\in\mathbb{Z}^n}e^{i2\pi\langle x,\xi \rangle }\widehat{f}(\xi),\,\,x\in\mathbb{T}^n.$$ 
The fundamental tool in the analysis of  periodic operators, that is, operators defined on function spaces on $\mathbb{T}^n,$ are the  H\"ormander class of periodic symbols. For every $m\in \mathbb{R},$ the periodic symbol of H\"ormander class of order $m$ and of $(\rho,\delta)$-type, denoted by $S^m_{\rho,\delta}(\mathbb{T}^n\times \mathbb{R}^n), \,\, 0\leq \rho,\delta\leq 1,$  consists of those functions $a(x,\xi)$ which are smooth in $(x,\xi)\in \mathbb{T}^n\times \mathbb{R}^n$ and which satisfy the following toroidal symbols inequalities
\begin{equation}\label{css}
|\partial^{\beta}_{x}\partial^{\alpha}_{\xi}a(x,\xi)|\leq C_{\alpha,\beta}\langle \xi \rangle^{m-\rho|\alpha|+\delta|\beta|}.
\end{equation}
It is well-known fact that Symbols in $S^m_{\rho,\delta}(\mathbb{T}^n\times \mathbb{R}^n)$ can be considered as symbols in $S^m_{\rho,\delta}(\mathbb{R}^n\times \mathbb{R}^n)$ (see \cite{Hor1, Ruz}) of order $m$ which are 1-periodic in $x.$
If $a(x,\xi)\in S^{m}_{\rho,\delta}(\mathbb{T}^n\times \mathbb{R}^n),$ the corresponding pseudo-differential operator is defined by
\begin{equation}\label{hh}
a(X,D_x)u(x)=\int\limits_{\mathbb{T}^n}\int\limits_{\mathbb{R}^n}e^{i2\pi\langle x-y,\xi \rangle}a(x,\xi)u(y)d\xi dy,\,\, u\in C^{\infty}(\mathbb{T}^n).
\end{equation}
The symbol class $S^m_{\rho,\delta}(\mathbb{T}^n\times \mathbb{Z}^n),\, 0\leq \rho,\delta\leq 1,$ consists of  those functions $a(x, \xi)$ which are smooth in $x$  for all $\xi\in\mathbb{Z}^n$ and satisfy

\begin{equation}\label{cs}
\forall \alpha,\beta\in\mathbb{N}^n,\exists\, C_{\alpha,\beta}>0,\,\, |\Delta^{\alpha}_{\xi}\partial^{\beta}_{x}a(x,\xi)|\leq C_{\alpha,\beta}\langle \xi \rangle^{m-\rho|\alpha|+\delta|\beta|}.
 \end{equation}

The operator $\Delta_\xi$ in  \eqref{cs} is the difference operator which is defined as follows: if $f:\mathbb{Z}^n\rightarrow \mathbb{C}$ is a discrete function and $(e_j)_{1\leq j\leq n}$ is the canonical basis of $\mathbb{R}^n,$ then $\Delta_\xi^\alpha$ is defined by
\begin{equation}
(\Delta_{\xi_{j}} f)(\xi)=f(\xi+e_{j})-f(\xi).
\end{equation}
If $k\in\mathbb{N},$ denote by $\Delta^k_{\xi_{j}}$  the composition of $\Delta_{\xi_{j}}$ with itself $k-$times. Finally, if $\alpha\in\mathbb{N}^n,$ $\Delta^{\alpha}_{\xi}= \Delta^{\alpha_1}_{\xi_{1}}\cdots \Delta^{\alpha_n}_{\xi_{n}}.$

The toroidal pseudo-differential operator (or periodic pseudo-differential operator) corresponding to the symbol $a(x,\xi) \in S^m_{\rho,\delta}(\mathbb{T}^n\times \mathbb{Z}^n) $ is defined by
\begin{equation}\label{aa}
a(x,D_x)u(x)=\sum_{\xi\in\mathbb{Z}^n}e^{i 2\pi\langle x,\xi\rangle}a(x,\xi)\widehat{u}(\xi),\,\, u\in C^{\infty}(\mathbb{T}^n).
\end{equation}
Similarly,  the H\"ormander classes on $\mathbb{Z}^n,$ (see Bothcway, Kibity and Ruzhansky \cite{Ruzhansky}) $S^m_{\rho,\delta}(\mathbb{Z}^n\times \mathbb{T}^n ),$ $0\leq \delta,\rho\leq 1,$ are defined by those functions $\sigma:\mathbb{Z}^n \times \mathbb{T}^n $ such that $ \sigma(n', \cdot) \in C^{\infty}(\mathbb{T}^n )$ for all $n' \in \mathbb{Z}^n,$ and for all multi-indices $\alpha, \beta$ there exists a constant $C_{\alpha, \beta}>0$ such that 
\begin{equation}\label{calderon3}
  |\partial_x^\beta\Delta_{n'}^\alpha \sigma(n', x)|\leq C_{\alpha,\beta}(1+|n'|)^{m-\rho|\alpha|+\delta|\beta|},
\end{equation} for all $n' \in \mathbb{Z}^n$ and for all $x \in \mathbb{T}^n.$
In this case we denote $\Psi^{m}_{\rho,\delta}(\mathbb{Z}^n\times \mathbb{T}^n)=\{t_\sigma:\sigma\in S^m_{\rho,\delta}(\mathbb{Z}^n\times \mathbb{T}^n )\}.$

The following proposition says that there exists a method  to interpolate the second argument of symbols on $\mathbb{T}^n\times \mathbb{Z}^n$ in a smooth way to get a symbol defined on $\mathbb{T}^n\times \mathbb{R}^n.$ The proof can be found in \cite[Theorem 4.5.3]{ Ruz}.
\begin{proposition}\label{eq}
Let $0\leq \delta \leq 1,$ $0< \rho\leq 1.$ The symbol $a\in S^m_{\rho,\delta}(\mathbb{T}^n\times \mathbb{Z}^n)$ if only if there exists  a Euclidean symbol $a'\in S^m_{\rho,\delta}(\mathbb{T}^n\times \mathbb{R}^n)$ such that $a=a'|_{\mathbb{T}^n\times \mathbb{Z}^n}.$
\end{proposition}

It is an interesting but non trivial fact that the definition of pseudo-differential operator on a torus \eqref{aa} given by Agranovich     and by H\"ormander \eqref{hh}  are equivalent. McLean  \cite{Mc} prove this for all  H\"ormander classes $S^m_{\rho,\delta}(\mathbb{T}^n\times \mathbb{Z}^n).$ A different proof to this fact can be found in \cite[Corollary 4.6.13]{Ruz}.

\begin{proposition}$\label{eqc}($Equality of Operators Classes$).$ For $0\leq \delta \leq 1,$ $0<\rho\leq 1,$ we have $\Psi^{m}_{\rho,\delta}(\mathbb{T}^n\times \mathbb{Z}^n)=\Psi^{m}_{\rho,\delta}(\mathbb{T}^n\times \mathbb{R}^n).$
\end{proposition}

In our further analysis it will be useful record the following fact: there exists a bijection  
\begin{equation}
i^*: S^{m}_{\rho,\delta}(\mathbb{T}^n\times \mathbb{Z}^n)\rightarrow S^{m}_{\rho,\delta}(\mathbb{T}^n\times \mathbb{R}^n)
\end{equation}
which associates  every symbol $\sigma\in S^{m}_{\rho,\delta}(\mathbb{T}^n\times \mathbb{Z}^n)$ to  an Euclidean symbol $ i^*\sigma\in S^{m}_{\rho,\delta}(\mathbb{T}^n\times \mathbb{R}^n)$ such that 
\begin{equation}
 \sigma(x,D_x)f(x)=(i^*\sigma)(X,D_x)f(x),\,\,\,f\in C^\infty(\mathbb{T}^n).
\end{equation} The existence of the mapping mapping  $i^*$ in \eqref{Eq:asym-homgtorus} is a natural consequence of Proposition \ref{eqc}. This can be constructed explicitly in terms of the periodisation technique developed by Ruzhansky and Turunen in \cite[Section 4.2]{Ruz}. 
\begin{definition}[Classical symbols on $\mathbb{T}^n$] A classical pseudo-differential operator $A=a(X,D_x)$ on $\mathbb{T}^n$ is an operator with symbol  $a\in S^{m}_{1,0}(\mathbb{T}^n\times\mathbb{R}^n)$, admitting an asymptotic expansion 
\begin{equation}\label{Eq:asym-homgtorus}
a(x,\xi)\sim \sum_{j=0}^{\infty} a_{m-j}(x,\xi)
\end{equation}
where $a_{m-j}(x,\xi)$ are homogeneous functions with respect to $\xi$ of degree $m-j$ for $\xi$ far from to  zero. 
\end{definition}
The set of classical symbols of order $m$ on $\mathbb{T}^n$ will be denoted  by $S^m_{cl}(\mathbb{T}^n\times \mathbb{R}^n)$ and the corresponding class of periodic pseudo-differential operators will be denoted by $\Psi^m_{cl}(\mathbb{T}^n)$.

Now, with the help of the classical symbols on $\mathbb{T}^n$, we introduce classical pseudo-differential operators on $\mathbb{Z}^n$. We start with the following definition of the classical symbols of $\mathbb{Z}^n$.
    \begin{definition}[Classical symbols on $\mathbb{Z}^n$]\label{CZN} A symbol $\sigma\in S^{m}_{1,0}(\mathbb{Z}^n\times \mathbb{T}^n)$ is called classical, if the symbol $\tau(x,k):=\overline{\sigma(-k,x)}$ satisfies that $i^*\tau\in S^m_{cl}(\mathbb{T}^n\times \mathbb{R}^n). $ \end{definition}
The set of classical symbols of order $m$ on $\mathbb{Z}^n$ will be denoted  by $S^m_{cl}(\mathbb{Z}^n\times \mathbb{T}^n)$ and the corresponding class of periodic pseudo-differential operators will be denoted by $\Psi^m_{cl}(\mathbb{Z}^n)$. 

Periodic and also discrete H\"ormander classes of pseudo-differential operators are closed under compositions, adjoints and parametrices. We refer the reader to  \cite{Ruzhansky,Ruz,Ruz-2} for details.

\section{Dixmier traces for pseudo-differential operators on $\mathbb{Z}^n$}\label{proof}
In this section, we present our main result. We begin this section by recalling some basic facts on Dixmier trace and related results. 

Let  $H$ be a Hilbert space and let $\mathcal{L}(H)$ be the algebra of all bounded linear operators on $H.$ We say that a  linear operator $A \in \mathcal{L}(H)$ is in the Dixmier class $\mathcal{L}^{(1,\infty)}(H)$ if 
\begin{equation}
\sum_{1\leq n\leq N}s_{n}(A)=O(\log(N)),\,\,\,N\rightarrow \infty,
\end{equation}
where $s(A):=\{s_{n}(A)\}_n$ denotes the sequence of singular values of $A$,  i.e. the square roots of the eigenvalues of the non-negative self-adjoint operator $A^\ast A$, and $s_{n}(A)$  are arranged in increasing order.
So, $\mathcal{L}^{(1,\infty)}(H)$ is endowed with the norm
\begin{equation}\label{dixmier}
\Vert A \Vert_{\mathcal{L}^{(1,\infty)}(H)}=\sup_{N\geq 2}\frac{1}{\log(N)}\sum_{1\leq n\leq N}s_{n}(A).
\end{equation}
The Dixmier trace $\textnormal{Tr}_{\omega}$ of a positive operator $A$ in $\mathcal{L}^{(1,\infty)}(H)$ can be formulated by using  an average function $\omega$ of the following form
$$\textnormal{Tr}_{\omega}(A)= \lim _{\omega }\frac{1}{\log(N)}\sum_{1\leq n\leq N}s_{n}(A),$$
where $\lim_{\omega}$ is  a positive linear functional on $l^{\infty}(\mathbb{N})$, the set of bounded sequences, satisfying: 
\begin{itemize}
\item $\lim_{\omega}(\alpha_n)\geq 0$ for all $\alpha_n\geq 0$;
\item $\lim_{\omega}(\alpha_n) = \lim(\alpha_n)$ whenever the ordinary limit exists;
\item $\lim_{\omega}(\alpha_1,\alpha_1, \alpha_2, \alpha_2,\alpha_3,\alpha_3\dots) = \lim_{w}(\alpha_n)$.
\end{itemize}
The functional $\textnormal{Tr}_{\omega}$ can be defined for positive operators, and later extended to the whole ideal $\mathcal{L}^{(1,\infty)}(H)$ by linearity (in this case $\textnormal{Tr}_\omega$ is not necessarily a positive functional). 

Let us assume that $M$ is a closed manifold. In the framework of noncommutative geometry, a remarkable result due to  Connes shows that any classical pseudo-differential operator, of order of the minus of the dimension of the manifold underliying, acting on $L^2(M)$ lies to the Dixmier measurable class, and moreover the Dixmier trace coincides with the noncommutative residue,  which is the unique trace (up to by factors) in the algebra of pseudo-differential operators with classical symbols on a closed manifold.

Now, we recall that  for a  compact manifold $M$ of dimension $\varkappa$ without boundary, a classical pseudo-differential operator $A$ on $M$ of order $m$, is defined by means of local symbols, this means that for any local chart $U$, the operator $A$ reads
$$Au(x)=\int\limits_{T^{*}_xU} e^{ix\xi}\sigma^A(x,\xi)\widehat{u}(\xi)\, d\xi$$
where  $\sigma^A(x,\xi)$  is a smooth function on $T^{*} U\cong U\times\mathbb{R}^n,$ $T^{*} _{x}U=\mathbb{R}^n$,  and admits an asymptotic expansion 
\begin{equation}\label{Eq:asym-homg}
\sigma^A(x,\xi)\sim \sum_{j=0}^{\infty}\sigma^{A}_{m-j}(x,\xi)
\end{equation}
where $\sigma^A_{m-j}(x,\xi)$ are homogeneous functions with respect to $\xi$ of degree $m-j$ for $\xi$ far away from the section zero. The set of classical pseudo-differential operators of order $m$ is denoted by $\Psi^m_{cl}(M)$. For $A\in\Psi_{cl}(M)$, for $x\in M$, $\int\limits_{|\xi|=1}\sigma_{-\varkappa}(x,\xi)\, d\Sigma(\xi),$ defines a local density which can be glued over $M$. So, a linear functional called the noncommutative residue is defined by 
\begin{equation}\label{resi}
\textnormal{res}\,(A)=\frac{1}{\varkappa(2\pi)^\varkappa}\int\limits_{M}\int\limits_{| \xi|=1}\sigma^{A}_{-\varkappa}(x,\xi)\, d\Sigma(\xi)\,d_{\textnormal{Vol}}x.
\end{equation}
which vanishes on non-integer order classical operators.  

In order to provide a trace formula for pseudo-differential operators on $\mathbb{Z}^n$ we will use the following identity in \cite{Ruzhansky}
\begin{equation}
t_\sigma=\mathscr{F}_{\mathbb{Z}^n}(\tau(x,D_x)^*)\mathscr{F}_{\mathbb{Z}^n}^{-1},
\end{equation}
and the following fundamental  result due to Connes \cite{Connes}.

\begin{theorem}[A. Connes]\label{connestheorem}
Let $M$ be a closed manifold of dimension $\varkappa.$ Then every classical and positive pseudo-differential operator $A$ of order $-\varkappa$ lies in $\mathcal{L}^{(1,\infty)}{(L^2(M))}$ and 
\begin{equation}
\textnormal{Tr}_{w}(A)=\textnormal{res}(A).
\end{equation}
\end{theorem}
\begin{corollary}\label{connestheorem2}
Let $\tau(x,D_x)$ be a classical positive pseudo-differential operator with symbol  $\tau\in S^{-n}_{1,0}(\mathbb{T}^n\times \mathbb{R}^n).$  Then $\tau(x, D_x)$ belongs to $\mathcal{L}^{(1,\infty)}{(L^2(\mathbb{T}^n))}$ and 
\begin{equation}
\textnormal{Tr}_{w}(\tau(x,D_x))=\textnormal{res}(\tau(x, D_x)).
\end{equation}
\end{corollary}

The following theorem is our main result on the Dixmier traceability of pseudo-differential operators on $\mathbb{Z}^n$. 

\begin{theorem}\label{mainTheorem}
Let us assume that $t_\sigma\in \Psi^{-n}_{cl}(\mathbb{Z}^n)$ is a positive  discrete pseudo-differential operator.  Then the operator $t_\sigma$ is Dixmier traceable on $\ell^2(\mathbb{Z}^n)$ and its Dixmier trace is given by
\begin{equation}
\textnormal{Tr}_\omega(t_\sigma)= \frac{1}{n(2\pi)^n} \int\limits_{\mathbb{T}^n}\int\limits_{\mathbb{S}^{n-1}}[i^*(\tau)(x,\xi)]_{(-n)}d\Sigma(\xi) d_{\textnormal{Vol}}x,\,\,\tau(x,k):=\overline{\sigma(-k,x)},
\end{equation} where $[i^*(\tau)(x,\xi)]_{(-n)}$ in the homogeneous term of order $-n$ is the asymptotic expansion of the symbol $i^*(\tau)(x,\xi)$ in homogeneous components as in \eqref{Eq:asym-homgtorus}.
\end{theorem}
\begin{proof}
Let us consider the following identity due to Botchway, Kibiti and Ruzhansky \cite{Ruzhansky}
\begin{equation}
t_\sigma=\mathscr{F}_{\mathbb{Z}^n}(\tau(x,D_x)^*)\mathscr{F}_{\mathbb{Z}^n}^{-1}.
\end{equation}
First we will prove that the set of the singular values of $t_\sigma$ agrees with the set of singular values of $\tau(x,D_x)^*$. We show this by using the fact that the spectrum is stable under unitary transformations. Since  the Fourier transform on $\mathbb{Z}^n$ is a unitary operator, we have that 
\begin{align*}
s(\tau(x,D_x)^*) &=\{\sqrt{\lambda}:\lambda\in \textnormal{Spectrum}(\tau(x,D_x) \tau(x,D_x)^*)\}\\&=\{\sqrt{\lambda}:\lambda\in\textnormal{Spectrum}( \mathscr{F}_{\mathbb{Z}^n}\tau(x,D_x) \tau(x,D_x)^* \mathscr{F}_{\mathbb{Z}^n}^{-1})\}\\
&=\{\sqrt{\lambda}:\lambda\in\textnormal{Spectrum}( \mathscr{F}_{\mathbb{Z}^n} \tau(x,D_x) \mathscr{F}_{\mathbb{Z}^n}^{-1}  \mathscr{F}_{\mathbb{Z}^n}   \tau(x,D_x)^* \mathscr{F}_{\mathbb{Z}^n}^{-1})\}\\
&=\{\sqrt{\lambda}:\lambda\in\textnormal{Spectrum}( \mathscr{F}_{\mathbb{Z}^n} \tau(x,D_x) \mathscr{F}_{\mathbb{Z}^n}^{-1}  \circ t_\sigma)\}\\
&=\{\sqrt{\lambda}:\lambda\in\textnormal{Spectrum}( t_\sigma^* t_\sigma)\}=s(t_\sigma),
\end{align*}
where in the last line we have used $t_\sigma^*=\mathscr{F}_{\mathbb{Z}^n} \tau(x,D_x) \mathscr{F}_{\mathbb{Z}^n}^{-1} .$ Now, it is clear that $t_\sigma $ is Dixmier traceable if and only if $\tau(x,D_x)^*$ is Dixmier traceable and in a such case $\textnormal{Tr}_{\omega}(t_\sigma)=\textnormal{Tr}_{\omega}(\tau(x,D_x)^*)$.
The positivity of $t_\sigma$ implies that $\tau(x,D_x)$ is also positive. Because $t_\sigma \in \Psi^{-n}_{cl}(\mathbb{Z}^n)$ we have that $\tau(x,D_x) \in \Psi^{-n}_{cl}(\mathbb{T}^n)$ and consequently  $\tau(x,D_x) $ is a compact operator on $L^{2}(\mathbb{T}^n),$ (see e.g.  \cite{Ruzhansky}, also \cite{s1} for $n=1$ ). In particular $\tau(x,D_x) $ is bounded and   taking into account its positivity we conclude that $\tau(x,D_x)$ is self-adjoint. So, if we assume the Dixmier traceability of $\tau(x,D_x)=\tau(x,D_x)^*, $ we deduce  $\textnormal{Tr}_{\omega}(\tau(x,D_x)^*)=\textnormal{Tr}_{\omega}(\tau(x,D_x)).$ Since $\tau(x,D_x) \in \Psi^{-n}_{cl}(\mathbb{T}^n)$ by the Connes Theorem \ref{connestheorem} in the form of Corollary \ref{connestheorem2}, we conclude that $\tau(x,D_x)$ is Dixmier traceable and its Dixmier trace can be computed from its noncommutative residue. So, we have
\begin{align*}
\textnormal{Tr}_{\omega}(\tau(x,D_x)) &=\textnormal{Tr}_{\omega}((i^{*}\tau)(X,D_x))=\textnormal{res}((i^{*}\tau)(X,D_x)\\
&= \frac{1}{n(2\pi)^n} \int\limits_{\mathbb{T}^n}\int\limits_{\mathbb{S}^{n-1}}[i^*(\tau)(x,\xi)]_{(-n)}d\Sigma(\xi)\,d_{\textnormal{Vol}}x.
\end{align*}
Consequently we obtain
\begin{equation}
\textnormal{Tr}_\omega(t_\sigma)= \frac{1}{n(2\pi)^n} \int\limits_{\mathbb{T}^n}\int\limits_{\mathbb{S}^{n-1}}[i^*(\tau)(x,\xi)]_{(-n)}d\Sigma(\xi)\,d_{\textnormal{Vol}}x.
\end{equation}
Thus, we finish the proof of the theorem.
\end{proof}

{\bf Acknowledgements.} This work was started when Duv\'an Cardona was supported by Pontificia Universidad Javeriana. C\'esar del Corral was supported by Universidad de Los Andes and Vishvesh Kumar was supported by CSIR India. 

Currently, Duv\'an Cardona and Vishvesh Kumar are supported by FWO Odysseus 1 grant G.0H94.18N: Analysis and Partial Differential Equations of Prof. Michael Ruzhansky. They thank Prof. Michael Ruzhansky for his support and encouragement.

Es con mucho aprecio que dedicamos \'este  trabajo  al Profesor L\'azaro Recht, uno de los mejores matem\'aticos argentinos. Su contribuci\'on a las generaciones matem\'aticas j\'ovenes, en Argentina, Venezuela y Colombia ha sido determinante, a\'unque \'el, modestia aparte, opine lo contrario.   
\bibliographystyle{amsplain}

\end{document}